\documentclass [a4paper,10pt]{amsart}
\usepackage{graphicx}
\usepackage[ansinew]{inputenc}    
\usepackage[all]{xy}

\usepackage{amssymb,amscd}
\newtheorem{thm}{Theorem}[section]

\def\C{\mathbb C}

\def\dim{\operatorname{dim}}

\usepackage{graphicx}
\usepackage[ansinew]{inputenc}  
\usepackage[all]{xy}
\usepackage{amssymb,amscd}
\usepackage{color}

\newtheorem{cor}[thm]{Corollary}
\newtheorem{teo}[thm]{Theorem}
\newtheorem{lem}[thm]{Lemma}

\theoremstyle{definition}
\newtheorem{ex}[thm]{Example}
\newtheorem{defi}[thm]{Definition}

\newtheorem{remark}[thm]{Remark}

\def\C{\mathbb C}

\def\dim{\operatorname{dim}}

\def\Aut{\operatorname{Aut}}

\begin{document}
	\title{Families of ICIS with constant total Milnor number}
	
	\author{R. S. Carvalho, J. J. Nuño-Ballesteros, B. Oréfice-Okamoto, J. N. Tomazella}
	
	\address{Departamento de Matem\'atica, Universidade Federal de S\~ao Carlos, Caixa Postal 676,
		13560-905, S\~ao Carlos, SP, BRAZIL}
	
	\email{rafaelasoares@dm.ufscar.br}
	
	\address{Departament de Matemàtiques, Universitat de València, Campus de Burjassot, 46100 Burjassot SPAIN}
	\address{Departamento de Matemática, Universidade Federal da Paraíba
		CEP 58051-900, João Pessoa - PB, Brazil}
	
	\email{Juan.Nuno@uv.es}
	
	\address{Departamento de Matem\'atica, Universidade Federal de S\~ao Carlos, Caixa Postal 676,
		13560-905, S\~ao Carlos, SP, BRAZIL}
	
	\email{brunaorefice@ufscar.br}
	
	\address{Departamento de Matem\'atica, Universidade Federal de S\~ao Carlos, Caixa Postal 676,
		13560-905, S\~ao Carlos, SP, BRAZIL}
	
	\email{jntomazella@ufscar.br}
	
	\thanks{The first author was partially supported by CAPES.
		The second author was partially supported by MICINN Grant PGC2018--094889--B--I00 and by GVA Grant AICO/2019/024 The third author was partially supported by FAPESP Grant
		2016/25730-0. The fourth author was partially supported by CNPq Grant
		309086/2017-5 and FAPESP Grant 2018/22090-5.}
		

\begin{abstract}
We show that a family of isolated complete intersection singularities (ICIS) with constant total Milnor number has no coalescence of singularities. This extends a well known result of Gabrielov, Lazzeri and Lê for hypersurfaces. We use A'Campo's theorem to see that the Lefschetz number of the generic monodromy of the ICIS is zero when the ICIS is singular. We give a pair applications for families of functions on ICIS which extend also some known results for functions on a smooth variety.
\end{abstract}

\maketitle

\section{Introduction}

A well known theorem proved independently  by Gabrielov, Lazzeri and Lê (see \cite{Gabrielov, Lazzeri, Le}) ensures that a family of hypersurfaces with constant total Milnor number has no coalescence of singularities. Also, Buchweitz and Greuel, in \cite{Buchweitz and Greuel}, present an example which shows that this is not true for families of reduced curves in general and they ask if it there is such an example for a family of isolated complete intersection singularities (ICIS). In this work, we answer negatively the Buchweitz and Greuel question.

More specifically, we consider an ICIS $(X,0)$ defined as the zero locus of a reduced holomorphic mapping $\phi:(\C^n,0)\to(\C^p,0)$. A deformation of $(X,0)$ is given by another holomorphic mapping $\Phi\colon(\C\times\C^n,0)\to (\C^p,0)$ such that $\Phi(0,x)=\phi(x)$ for all $x$. We fix a representative $\Phi:D\times W\to\C^p$, where $D$ and $W$ are open neighbourhoods of the origin in $\C$ and $\C^n$ respectively. For each $t\in D$, we
put $\phi_t(x):=\Phi(t,x)$ and $X_t:=\phi_t^{-1}(0)$ and call $(X_t)_{t\in D}$ a family of ICIS.

Let $B_\epsilon\subset W$ be a  Milnor ball for $X$ at $0$, that is, $0$ is the only singular point of $X$ on $B_\epsilon$ and for all $\epsilon'$ with $0<\epsilon'\le \epsilon$, $\partial B_{\epsilon'}$ is transverse to $X$. After shrinking $D$ if necessary, we can assume that $\partial B_\epsilon$ is also transverse to $X_t$, for all $t\in D$. This forces that $X_t$ has a finite number of singular points $z_1(t),\dots,z_k(t)$ contained in $B_\epsilon$. The \emph{total Milnor number} of $X_t$ is defined as
\[
\mu_t:=\sum_{i=1}^{k_t} \mu(X_t, z_i(t)),
\]
where $\mu(X_t; z_i(t))$ is the Milnor number of $X_t$ at $z_i(t)$. We say that the family has \emph{constant total Milnor number} if $\mu_t=\mu_0$, for all $t\in D$. Our main result is that a family of ICIS with constant total Milnor number has no coalescence of singularities as follows:

\begin{teo}\label{main} Let $(X_t)_{t\in D}$ be a family of ICIS with constant total Milnor number. Then $X_t$ has a unique singular point $z_1(t)$ on $B_\epsilon$ and $\mu(X_t,z_1(t))=\mu(X,0)$, for all $t\in D$.
\end{teo}

As we mentioned at the beginning, the hypersurface case of Theorem \ref{main} ($p=1$) gives the result of Gabrielov, Lazzieri and Lê in \cite{Gabrielov, Lazzeri, Le}. Here we follow the proof of Lê in \cite{Le}, which is based on A'Campo's theorem that the Lefschetz number of the monodromy of a hypersurface with an isolated singular point is zero \cite{A'campo}. In fact, A'Campo's theorem is more general and can be used for the monodromy of any holomorphic function $f:(X,0)\to(\C,0)$ on a complex analytic variety $X$ such that $f\in\mathfrak m_{X,0}^2$, where we denote by $\mathfrak m_{X,0}$ the maximal ideal of the local ring $\mathcal O_{X,0}$. A recent paper \cite{GLN} also shows that in these conditions, the geometric monodromy can be constructed with no fixed points, which is even stronger than A'Campo's theorem.

When $(X,0)$ is an ICIS of dimension $d$, we can choose generic coordinates in $\C^p$ in such a way that the zero locus $(X_1,0)$ of $\phi_1,\dots,\phi_{p-1}$ is also an ICIS of dimension $d+1$. We can imitate the hypersurface case by seeing $(X,0)$ as the special fibre of the restriction $\phi_p:(X_1,0)\to(\C,0)$. We show that the monodromy of $\phi_p:(X_1,0)\to(\C,0)$ is independent of the choice of the generic coordinates and call it the \emph{generic monodromy} of $(X,0)$. Moreover, we also show that if $0$ is a singular point of $X$, then we can choose the generic coordinates such that $\phi_p\in\mathfrak m_{X_1,0}^2$, so we can use A'Campo's theorem in this situation.

In the last section, we give two applications of Theorem \ref{main} for families of functions $f_t:(X,0)\to(\C,0)$, $t\in D$, on a fixed ICIS $(X,0)$. We denote by 
$S(f_t)$ the set of critical points of $f_t$ and by $S(Y_t)$ the set of singular points of the fibre $Y_t:=f_t^{-1}(0)$. Obviously, $S(Y_t)$ is contained in $S(f_t)$, but $S(f_t)$ may have critical points contained in fibres different from $Y_t$. We show that if $S(f_t)=S(Y_t)$, then $S(f_t)=\{0\}$. Again this result appears in the papers of Gabrielov, Lazzeri and Lê \cite{Gabrielov, Lazzeri, Le} when $X=\C^n$.

For the second application we consider
$J=J(f_t,\phi)$ the relative Jacobian ideal, that is, the ideal in $\mathcal O_{\C\times X,0}$ generated the maximal minors of the Jacobian matrix of $(f_t,\phi)$ with respect to the variables $x_1,\dots,x_n$ in $\C^n$ (that is, we do not consider here derivatives with respect to $t$). The zero locus of $J$ in $(\C\times X,0)$ is the set germ of points $(t,x)$ such that $x\in S(f_t)$. We show that if $\partial F/\partial t$ belongs to the radical of $J$ in $\mathcal O_{\C\times X,0}$, with $F(x,t)=f_t(x)$, then the zero locus of $J$ is equal to $(\C\times\{0\},0)$. This result completes some of the equivalences considered in a previous paper \cite{Rafaela Bruna e Tomazella}  about families of functions on ICIS. Such equivalences where showed by Greuel in \cite{Greuel} again in the case $X=\C^n$.

\section{The generic monodromy of an ICIS}

We recall some basic facts about the monodromy of an ICIS.
Let $\phi\colon(\mathbb{C}^n,0)\to
(\mathbb{C}^p,0)$ be a holomorphic map germ such that the germ defined by its zero set, $(X,0)\subset(\mathbb{C}^n,0)$, is an ICIS of dimension $d$.  We denote by $\Sigma$ the set germ of critical points of $\phi$ and by $\Delta=\phi(\Sigma)$ its discriminant. It is well known that $\Delta$ is a hypersurface in $(\C^p,0)$ (see for instance \cite{looijenga}). We choose $0<\delta\ll \epsilon\ll 1$ such that the restriction
\begin{equation}\label{fibration}
\phi: B_\epsilon\cap \phi^{-1}(\mathring B_\delta\smallsetminus \Delta)\to \mathring B_\delta\smallsetminus \Delta
\end{equation}
is a proper submersion, where $B_\epsilon$ is the closed ball of radius $\epsilon$ centered at the origin in $\C^n$ and $\mathring B_\delta$ is the open ball of radius $\delta$ centered at the origin in $\C^p$. After shrinking $\delta$ if necessary, $\mathring B_\delta\smallsetminus \Delta$ is connected. By the Ehresmann Lemma for manifolds with boundary (see \cite[Lemma 6.2.10]{LNS}), \eqref{fibration} is a smooth fibre bundle with fibre $F=B_\epsilon\cap \phi^{-1}(y)$, for some $y\in \mathring B_\delta\smallsetminus \Delta$. The fibre $F$ is called the \emph{Milnor fibre} of $(X,0)$. 

Associated with the fibre bundle \eqref{fibration}, we also have the \emph{monodromy action} which is the group homomorphism
\begin{equation}\label{monodromy}
\pi_1(\mathring B_\delta\smallsetminus \Delta; y)\longrightarrow \Aut(H_*(F))
\end{equation}
where $\Aut(H_*(F))$ is the group of automorphisms of the homology of the fibre $H_*(F)$. The image of \eqref{monodromy} is known as the \emph{monodromy group} of $(X,0)$.

Since $\mathring B_\delta\smallsetminus \Delta$ is connected, we may write $\pi_1(\mathring B_\delta\smallsetminus \Delta)$ instead of $\pi_1(\mathring B_\delta\smallsetminus \Delta;y)$. A different choice of base point will give a conjugated monodromy via the isomorphism induced by a path between the two base points.

\begin{defi} We say that a line $L$ in $\C^p$ through the origin is \emph{generic} if, as set germs,
\[
L\cap \Delta=\{0\},\quad L\cap C_0(\Delta)=\{0\},
\] 
where $C_0(\Delta)$ is the Zariski tangent cone of $\Delta$ at the origin. The \emph{monodromy with respect to $L$} is the monodromy $H_*(F)\to H_*(F)$ associated with the smooth fibre bundle
\begin{equation}\label{fibration2}
\phi: B_\epsilon\cap \phi^{-1}(\partial B_\eta\cap L)\to \partial B_\eta\cap L,
\end{equation}
for $\eta>0$ small enough, and induced by the loop $\alpha_{y}(\theta)=e^{2\pi i\theta}y$, with $\theta\in[0,1]$ and some $y\in \partial B_\eta\cap L$.
\end{defi}

It follows from the Noether Normalisation Theorem (see e.g. \cite[Corollary 3.3.19]{deJongPfister}) that the set $\Omega$ of generic lines is a non-empty Zariski open subset of $\mathbb P^{p-1}$. In particular, $\Omega$ is connected.

On the other hand, the condition that $L$ is generic is equivalent to the fact that the local intersection number $i(\Delta,L;0)$ of $\Delta$ and $L$ in $\C^p$ at the origin is equal to the multiplicity $m:=m(\Delta,0)$. In fact, let $\varphi\in\mathcal O_p$ be such that $\varphi=0$ is a reduced equation for $\Delta$. Then we can write $\varphi=\varphi_m+\varphi_{m+1}+\dots$, where each $\varphi_i$ is homogeneous of degree $i$ and $\varphi_m\ne0$. Moreover, $\varphi_m=0$ is an equation for $C_0(\Delta)$. Assume also that $L$ is generated by some $v\in\C^p$, $v\ne0$, so it has a parametrisation $\gamma(t)=tv$, $t\in\C$. This gives that $i(\Delta,L;0)$ is the order of the composition $\gamma\circ\varphi:(\C,0)\to(\C,0)$ and this is equal to $m$ if and only if $L$ is generic.

\begin{lem} The monodromy with respect to $L$ is independent of the choice of the generic line $L\in\Omega$.
\end{lem}

\begin{proof} Since $\Omega$ is connected, it is enough to show that given $L_0\in\Omega$, there exists an open neighbourhood $W$ of $L_0$ in $\Omega$ such that the monodromy with respect to $L$ is independent of $L\in W$. For simplicity, we assume that $L_0$ has homogeneous coordinates $[u_0:1]$ in $\mathbb P^{p-1}$, for some $u_0\in\C^{p-1}$. We take an open neighbourhood $U$ of $u_0$ in $\C^{p-1}$ such that for all $u\in U$, the line $L_u=[u:1]\in\Omega$. 

Assume that $\Delta$ has reduced equation $\varphi=0$, for some $\varphi\in\mathcal O_p$. Let $\gamma_u:\C\to\C^p$ be the parametrisation of $L_u$  given by $\gamma_u(t)=t(u,1)$. Then the local intersection number of $L_u$ and $\Delta$ at the origin is 
\[
i(\Delta,L_u;0)=\dim_\C\frac{\mathcal O_{\C,0}}{(\varphi\circ\gamma_u)},
\]
which is constant equal to $m:=m(\Delta,0)$ for all $u\in U$. Since
\[
\frac{\mathcal O_{\C\times U,(0,u_0)}}{(\varphi\circ\gamma_u)},
\]
is Cohen-Macaulay of dimension $p-1$, we have conservation of the multiplicity. This means that there exists an open neighbourhood $D\times \mathring B_\rho$ of  $(0,u_0)$ in $\C\times U$  such that for all $u\in \mathring B_\rho$,
\[
\dim_\C\frac{\mathcal O_{\C,0}}{(\varphi\circ\gamma_{u_0})}=\sum_{\varphi\circ\gamma_{u}(t)=0}\dim_\C\frac{\mathcal O_{\C,t}}{(\varphi\circ\gamma_{u})}.
\]
Now the constancy of the local intersection number at the origin implies that $\varphi\circ\gamma_u(t)=0$ if and only if $t=0$. Hence we can find $\eta>0$ small enough such that $\partial B_\eta\cap L_u\cap \Delta=\emptyset$, for all $u\in \mathring B_\rho$. 

Finally, let 
\[
A=\partial B_\eta\cap\left(\bigcup_{u\in \mathring B_\rho} L_u\right).
\]
We have $A\subset \mathring B_\delta\smallsetminus \Delta$ and $A$ is homeomorphic to $S^1\times \mathring B_\rho$. Hence $A$ has the homotopy type of $S^1$ and all the loops $\alpha_{y}$, for $y\in A$, are homotopic in $B_\delta\smallsetminus \Delta$, via the isomorphism induced by a path between the base points. By \eqref{monodromy}, the associated monodromies $H_*(F)\to H_*(F)$ are all equal. We take $W$ as the set of lines $L_u$, with $u\in \mathring B_\rho$.
\end{proof}

\begin{defi} The monodromy with respect to a generic line $L$ is called the \emph{generic monodromy} of $(X,0)$.
\end{defi}

In practice, given a generic line $L$ we can choose coordinates in $\C^p$ such that $L$ is the line given by the $y_p$-axis. It follows that the set germ $(X_1,0)$ defined by the zeros of $\phi_1,\dots,\phi_{p-1}$ is an ICIS of dimension $d+1$. The generic monodromy is now the monodromy of the restriction $\phi_p|_{(X_1,0)}:(X_1,0)\to(\C,0)$, that is, the monodromy associated with the smooth fibre bundle
\[
\phi_p: X_1\cap B_\epsilon\cap \phi_p^{-1}(\partial D_\eta)\to \partial D_\eta,
\]
for $0<\eta\ll \epsilon\ll 1$.

\begin{lem}\label{acampo} If $(X,0)$ is singular, then its generic monodromy has Lefschetz number equal to zero.
\end{lem}

\begin{proof} Since $(X,0)$ is singular, $\phi:(\C^n,0)\to(\C^p,0)$ has rank $r<p$. We can choose coordinates $(x,y)$ in $\C^n=\C^r\times \C^{n-r}$ and $u$ in $\C^{p-r}$ such that
\[
\phi(x,y)=(x,\phi_{1}(x,y),\dots,\phi_{p-r}(x,y)),
\]
with $\phi_i\in\mathfrak m_n^2$, for $i=1,\dots,p-r$. Hence, $(X,0)=(\{0\}\times Y,0)$, where $(Y,0)$ is the ICIS defined as the zero set of $\psi:(\C^{n-r},0)\to(\C^{p-r},0)$ given by
\[
\psi(y)=(\phi_{1}(0,y),\dots,\phi_{p-r}(0,y)).
\]
After a new linear coordinate change in $\C^{p-r}$ if necessary, we can assume that the $u_{p-r}$-axis is a generic line in $\C^{p-r}$ for $(Y,0)$. Hence, the $u_{p-r}$-axis is also a generic line in $\C^{p}$ for $(X,0)$. The generic monodromy of $(X,0)$ is the monodromy of the restriction $\phi_{p-r}|_{(X_1,0)}:(X_1,0)\to(\C,0)$, where $(X_1,0)$ is the ICIS given by $x=0$ and $\phi_1(0,y)=\dots=\phi_{p-r-1}(0,y)=0$. Since $\phi_{p-r}\in\mathfrak m_n^2$, its restriction belongs to $\mathfrak m_{X_1,0}^2$. By A'Campo's theorem \cite[Theorem 1 bis]{A'campo} (see also \cite[Theorem 0.2]{GLN}), the monodromy has Lefschetz number equal to zero.
\end{proof}

\section{Main theorem}

We consider an ICIS $(X,0)$ defined as the zero locus of a reduced holomorphic mapping $\phi:(\C^n,0)\to(\C^p,0)$. A deformation of $(X,0)$ is given by another holomorphic mapping $\Phi\colon(\C\times\C^n,0)\to (\C^p,0)$ such that $\Phi(0,x)=\phi(x)$ for all $x$. We fix a representative $\Phi:D\times W\to\C^p$, where $D$ and $W$ are open neighbourhoods of the origin in $\C$ and $\C^n$ respectively. For each $t\in D$, we
put $\phi_t(x):=\Phi(t,x)$ and $X_t:=\phi_t^{-1}(0)$ and call $(X_t)_{t\in D}$ a family of ICIS.

\begin{teo}\label{unico ponto critico} Let $(X_t)_{t\in D}$ be a holomorphic family of complete intersections as described above. Let $B$ be a Milnor ball for $(X,0)$. Then for $t$ different of zero and small enough, $X_t\cap B$ has $k_t$ singular points:  $z_1(t),\ldots,z_{k_t}(t)$. If $$\mu(X,0)=\displaystyle\sum_{i=1}^{k_t}\mu(X_t,z_i(t)),$$
 then $k_t=1$ and $\mu(X_t,z_1(t))=\mu(X,0)$ for all $t$ small enough. \end{teo}

\begin{proof} We choose a Milnor ball $ B_i\subset \mathring B$ for each $(X_t,z_i)$, $i=1,\ldots,k_t$.
	
	Let $X_{t,A}:=(\phi_t)^{-1}(A)\cap B$ be a generic fiber of $\phi_t$. 
	We define the sets $U=X_{t,A}\cap (\cup_{i=1}^k B_i)$ and  $V=X_{t,A}\cap ( B\setminus\cup_{i=1}^k \mathring B_i)$.

	 We consider the following piece of the Mayer-Vietoris sequence in homology of the pair $(U,V)$ in $X_{t,A}$ 
	 
	 \begin{center} $\cdots\rightarrow
	 	H_{d+1}(X_{t,A}){\rightarrow}H_{d}(U\cap V)\rightarrow H_d(U)\oplus H_{d}(V)\rightarrow H_{d}(X_{t,A})\rightarrow H_{d-1}(U\cap V){\rightarrow}H_{d-1}(U)\oplus H_{d-1}(V){\rightarrow}H_{d-1}(X_{t,A}){\rightarrow}
	 	\cdots$ \end{center}
	where $d=n-p$. 
	
	Since the Milnor fiber of an ICIS of dimension $d$ is a bouquet of $d$-spheres, we conclude that $H_{d-1}(X_{t,A})=H_{d-1}(U)=H_{d+1}(X_{t,A})=0$. Therefore the sequence
	
	\begin{center} $
		0{\rightarrow}H_{d}(U\cap V)\rightarrow \oplus_{i=1}^{k_t} H_{d}(X_{t,A}\cap B_i)\oplus H_{d}(V)\rightarrow H_{d}(X_{t,A})\rightarrow H_{d-1}(U\cap V){\rightarrow} H_{d-1}(V){\rightarrow}0$ \end{center}
	 is exact. Hence $a_1-a_2+\mu(X,0)-\sum_{i=1}^{k_t} \mu(X_t,z_i)-a_3+a_4=0$, where $a_1=\dim_{\mathbb{C}}H_{d-1}(V)$, $a_2=\dim_{\mathbb{C}}H_{d-1}(U\cap V)$, $a_3=\dim_{\mathbb{C}}H_{d}(V)$ and $a_4=\dim_{\mathbb{C}}H_{d}(U\cap V)$. By the hypothesis, $\mu(X,0)=\displaystyle\sum_{i=1}^{k_t}\mu(X_t,z_i)$. Therefore 
	 \begin{equation}\label{eq1}
	 a_1-a_2-a_3+a_4=0.
	 \end{equation}
	
	Applying the generic monodromy to the exact sequence, we get the commutative diagram\newline
	$\label{diagrama} \xymatrix{ 0\to H_{d}(U\cap V)\ar[d]^{Id_4}\ar[r] &
			\oplus_{i=1}^{k_t}H_{d}(X_{t,A}\cap B_i)\oplus H_d(V)\ar[d]^{\oplus_{i=1}^k h_i^*\oplus Id_3}\ar[r]\!& H_{d}(X_{t,A})\ar[d]^{h^*}\ar[r]\! &
		 H_{d-1}(U\cap V)\ar[d]^{Id_2}\ar[r]\! & H_{d-1}(V)\ar[d]^{Id_1}\to 0\\
		 0\to H_{d}(U\cap V)\ar[r] &
		 \oplus_{i=1}^{k_t}H_{d}(X_{t,A}\cap B_i)\oplus H_d(V)\ar[r]\!& H_{d}(X_{t,A})\ar[r]\! &
		 H_{d-1}(U\cap V)\ar[r]\!&H_{d-1}(V)\to 0}$
		 where $h^*$ and $h_i^*$ are the generic monodromies of $(X,0)$ and $(X_t,z_i)$, respectively, and $Id_1, \dots, Id_4$ are the identities.

	Hence, if we denote by $tr$ the trace of a homomorphism then we have $tr(Id_1)-tr(Id_2)+tr(h^*)-(\sum_{i=1}^{k_t}tr(h_i^*)+tr(Id_3))+tr(Id_4)=0.$ That is, 
	\begin{equation}\label{eq2} a_1-a_2+tr(h^*)-\sum_{i=1}^{k_t}tr(h_i^*)-a_3+a_4=0. \end{equation}
	
	By comparing the equations (\ref{eq1}) and (\ref{eq2}), we conclude that
	\begin{equation}\label{eq3}
	tr(h^*)=\sum_{i=1}^{k_t}tr(h_i^*).
	\end{equation}
	
	On the other hand, by Lemma \ref{acampo},
		
		\begin{equation*} \sum_{q\geq 0}(-1)^q tr({h^*}^q)=0 \, \, \mbox{and} \, \,  \sum_{q\geq 0}(-1)^q tr({h_i^*}^q)=0 \end{equation*}

		Since ${h^*}^q=0$ and ${h_i^*}^q=0$, for all $q\neq 0, d$, and $tr({h^*}^0)=1=tr({h_i^*}^0)$, we conclude that $$0=1+(-1)^dtr({h^*}^d) \, \, \mbox{and} \, \, 0=1+(-1)^dtr({h_i^*}^d).$$ 
		That is, $tr({h^*}^d)=tr({h_i^*}^d)=(-1)^{d+1}$. 
		
		Then, by the equation (\ref{eq3}), $k_t=1$ and thus $\mu(X,0)=\mu(X_t,z_1)$. \end{proof}
	\begin{remark}
		 The hypersurface case of the previous theorem ($p=1$) gives the result of Gabrielov, Lazzieri and Lê in \cite{Gabrielov, Lazzeri, Le}.
		It  answers negatively the question formulated by Buchweitz and Greuel in the Example 7.2.5 of \cite{Buchweitz and Greuel} about the existence of a family of ICIS with constant total Milnor number with splitting of the singularity.
	\end{remark}

\section{Applications}
\subsection{}
We also know, from the works of Lê, Lazzeri and Gabriélov (\cite{Le,Lazzeri,Gabrielov}) that if a family of holomorphic function germs $f_t\colon(\C^n,0)\to(\C,0)$ with isolated singularity is such that  the singular set of $f_t$ is equal to the singular set of the variety $f_t^{-1}(0)$ then, in a small ball around the origin in $\C^n$, the only singular point of each $f_t$ is $0$.
We can use the result of the previous section to show that this continues to be true if we change the source of the germs by an ICIS.

More specifically, we consider an ICIS $(X,0)$ defined as the zero locus of a reduced holomorphic mapping $\phi:(\C^n,0)\to(\C^p,0)$ and $f\colon(X,0)\to(\C,0)$ a holomorphic function germ with isolated singularity. In this case, we define the Milnor number of $f$, $\mu(f)$, as the dimension as a $\C$-vector space of the quotient $\frac{\mathcal O_n}{\left\langle\phi\right\rangle+J(f,\phi)}$, where $J(f,\phi)$ is the ideal generated by the maximal minors of the Jacobian matrix of the map $(f,\phi)$.

Let $F:(\C\times X,0)\rightarrow
(\mathbb{C},0)$ be a (flat) deformation of $f$. We fix a representative $F:D\times W\rightarrow\C$, where $D$ and $W$ are open neighbourhoods of the origin in $\C$ and $X$ respectively. We denote $f_t(x):=F(t,x)$ and $Y_t=X\cap f_t^{-1}(0)$. In this case, the singular set of each $f_t$, $S(f_t)$, is the zero set of the ideal in $\mathcal O_W$ (the ring of holomorphic functions from $W$ to $\C$) generated by $\phi$ and by the minors of maximal size of the Jacobian matrix of the map $(f_t,\phi)$; the singular set of $Y_t$, $S(Y_t)$ is the zero set of the ideal generated by $\phi$, $f_t$ and the minors of maximal size of the Jacobian matrix of the map $(f_t,\phi)$ (we remark that we are not making derivatives with respect to $t$). It is always true that $S(Y_t)\subset S(f_t)$. We assume $0$ is a singular point of  $S(f_t)$
\begin{cor}\label{conj singular da funcao}
If $S(f_t)=S(Y_t)$, then $S(f_t)=\{0\}$  on a fixed neighborhood of the origin, which is independent of $t$.
\end{cor}

\begin{proof} 
We will consider the Milnor number of each germ $f_t\colon(X,0)\to \C$. In this case,
$$\mu(f_t)=\dim_{\C}\frac{\mathcal O_n}{\left\langle\phi\right\rangle+J(f_t,\phi)}.$$

We choose a Milnor ball $B$ for $(X,0)$. Since $\frac{\mathcal{O}_{n+1}}{\langle \phi\rangle+J(f_t,\phi)}$ is a Cohen-Macaulay ring  (see \cite[Lemma 1.9]{Thesis Greuel}) (here, we are considering $t$ as variable), by the conservation of number principle, in $\mathring B$,
$$
\mu(f)=\displaystyle\sum_{x\in S(f_t)}\mu({f_t},x),
$$ 
for $t$ sufficiently small. Therefore, since we are assuming that $S(f_t)=S(Y_t)$, 
$$
\mu(f)=\displaystyle\sum_{x\in S(Y_t)}\mu({f_t},x).
$$
By using the Lê-Greuel formula (see \cite{Thesis Greuel, Le 1}) we see that
$$\mu(X,0)+\mu(X\cap f^{-1}(0),0)=\sum_{x\in S(Y_t)}(\mu(X,x)+\mu(Y_t,x)).$$	
Since $X$ has an isolated singularity at the origin, $$\mu(X,0)+\mu(X\cap f^{-1}(0),0)=\mu(X,0)+\sum_{x\in S(Y_t)}\mu(Y_t,x).$$	
Hence, $$\mu(X\cap f^{-1}(0),0)=\sum_{x\in S(Y_t)}\mu(Y_t,x).$$
By Theorem \ref{unico ponto critico}, $\sharp S(Y_t)=1$ and, therefore, $S(f_t)=\{0\}$. 
	
\end{proof}

\subsection{}

Let $f_t\colon(\C^n,0)\to\C$ be a holomorphic family of holomorphic function germs with isolated singularities. In \cite{Greuel}, Greuel presents five conditions which are equivalent to such a family to have constant Milnor number. These conditions are related to the integral closure and to the radical of the ideal which defines the singular set of the members of the family.

In \cite{Rafaela Bruna e Tomazella}, we study what happens to Greuel's result for a holomorphic family of holomorphic function germs $f_t\colon(X,0)\to \C$, where $(X,0)$ is an ICIS. Let $\phi:(\mathbb{C}^n,0)\rightarrow
(\mathbb{C}^p,0)$ be a holomorphic map germ whose zero set defines an ICIS $(X,0)$. Let $f\colon (X,0)\to(\C,0)$ be a holomorphic function germ with isolated singularity and $F\colon(\C\times X,0)\to
(\mathbb{C},0)$ be a (flat) deformation of $f$. We write $f_t(x):=F(x,t)$. We denote by $J$ the ideal in $\mathcal O_{\C\times X}$ whose zero set is the union of the singular sets of the $f_t$'s, that is $J=J(f_t,\phi)$ (without derivatives with respect to $t$). The assertions of Greuel's theorem for this case are
\begin{enumerate}
\item\label{1} $F$ is $\mu$-constant;
\item\label{2} For every holomorphic curve $\gamma: (\mathbb{C},0)\rightarrow
(\mathbb{C}\times X,0)$, $\nu\left(\frac{\partial F}{\partial
	t}\circ\gamma\right)>\inf\{\nu(g_i\circ\gamma)\}$, where $g_i$ is each of the maximal minors of the Jacobian matrix of $(f_t,\phi)$ and $\nu$ denotes the usual valuation of a complex curve;
\item\label{3} Same statement as in (\ref{2}) with $``>"$ replaced by $``\geq"$;
\item\label{4} $\frac{\partial F}{\partial t} \in \overline{J}$ as an ideal in $\mathcal{O}_{\mathbb{C}\times X}$, where $\overline{J}$ is the integral closure of the ideal;
\item\label{5} $\frac{\partial F}{\partial t} \in \sqrt{J}$ as an ideal in
$\mathcal{O}_{\mathbb{C}\times X}$;
\item\label{6} The zero set of $J$, $v(J)$ is equal to $\mathbb{C}\times\{0\}$ near $(0,0)$ in
$\mathcal{O}_{\mathbb{C}\times X}$.
\end{enumerate}

We prove, in \cite{Rafaela Bruna e Tomazella}, that
\begin{center} (\ref{2}) $\Rightarrow$ (\ref{3}) $\Leftrightarrow$ (\ref{4})	$\Rightarrow$ (\ref{5}),

	(\ref{4}) $\Rightarrow$ (\ref{1}) $\Leftrightarrow$ (\ref{6}) $\Rightarrow$ (\ref{5}).
	
\end{center}

Also in \cite{Rafaela Bruna e Tomazella} we present the following example (Example 3.6(b))

\begin{ex} Let $(X,0)\subset (\mathbb{C}^2,0)$ be defined by $\phi(x,y)=x^p-y^q$, with $q\geq 3$, and $f\colon(X,0)\to(\mathbb{C},0)$ defined by 	$f(x,y)=x$. We consider the deformation of $f$ defined by $F(t,(x,y))=x+ty$.
	
In this case $J=\langle -qy^{q-1}-ptx^{p-1}\rangle$. If $p<q$ then it is not hard to see that $\mu(f)=pq-p$ and $\mu(f_t)=pq-q$. Therefore (\ref{1}) is not true.
\end{ex}

Moreover, we said wrongly in \cite{Rafaela Bruna e Tomazella} that $\frac{\partial F}{\partial t} \in \sqrt{J}$ as an ideal in $\mathcal{O}_{\mathbb{C}\times X}$. If it was the case, $(\ref{5})$ would be true for this example and we would have a counter-example for $(\ref{5})\Rightarrow (\ref{1})$. We correct this mistake in the following theorem which proves that $(\ref{5})\Rightarrow (\ref{6})$ and, therefore, $(\ref{5})\Rightarrow (\ref{1})$. Here, we see $t$ as a variable.

\begin{teo} If $\frac{\partial F}{\partial t} \in \sqrt{J}$ as an ideal in $\mathcal{O}_{\mathbb{C}\times X}$ then $v(J)=\mathbb{C}\times\{0\}$ near $(0,0)$. \end{teo}

\begin{proof} Since $\frac{\partial F}{\partial t} \in \sqrt{J}$ as an ideal in $\mathcal{O}_{\mathbb{C}\times X}$ then $\frac{\partial F}{\partial t} \in \sqrt{\left\langle \phi\right\rangle+J}$ as an ideal in $\mathcal{O}_{n+1}$. Therefore, by Hilbert's Nullstellensatz Theorem, $v(\sqrt{\left\langle \phi\right\rangle+J})\subset v(\frac{\partial F}{\partial t})$. Then, 
$$
\frac{\partial F}{\partial
		t}|_{v(\left\langle \phi\right\rangle+J)}\equiv 0.$$
	
We remark that $v(\left\langle\phi\right\rangle+J)$ is the set $\{(t,x):x\in S(f_t)\}$. Let $(t,x)\in v(\left\langle\phi\right\rangle+J)$.

On one hand, if $x\in S(X)$ then $x=0$ because $(X,0)$ has isolated singularity at $0$, therefore, $f_t(x)=f_t(0)=0$.

 On the other hand, if $x\notin S(X)$ then there exists a diffeomorphism $\psi\colon \C\times U\to \C\times U'$, where $U$ and $U'$ are open neighbourhoods in $\mathbb{C}^{d}$ and $X$ respectively with $0\in U$ and $x\in U'$ such that $\psi(t,0)=(t,x)$.

We denote by $J'$ the pullback of $J$ by $\psi$, $J'=	\psi^*(J)$. By the hypothesis, $\psi^*(\frac{\partial F}{\partial t}) \in \sqrt{J'}$ and, therefore $\frac{\partial (\Psi^*F)}{\partial t} \in \sqrt{J'}$, as an ideal in 	$\mathcal{O}_{d+1,0}$. Then, $\frac{\partial (\Psi^*F)}{\partial t}|_{V(J')}\equiv 0$ and thus $\Psi^*F$ does not depend of $t$ in $J'$, therefore $\Psi^*F|_{J'}\equiv 0$. Hence, $f_t(x)=F(t,x)=F(\psi(t,0))=\psi^*F(t,0)=0$.

Hence, the critical set of $f_t$ is in $f_t^{-1}(0)$. Therefore, by Corollary \ref{conj singular da funcao}, $S(f_t)=\{0\}$.
	
	 \end{proof}

\section*{Acknowledgments}
	
We are grateful to Gert-Martin Greuel for all the comments and suggestions in a previous version of the text.

\bibliographystyle{amsplain}

\end{document}